\documentclass[12pt]{amsart}
\usepackage[normalem]{ulem}

\usepackage{amssymb}

\allowdisplaybreaks

\numberwithin{equation}{section}

\newtheorem{theorem}{Theorem}[section]
\newtheorem{proposition}{Proposition}[section]
\newtheorem{lemma}[theorem]{Lemma}

\theoremstyle{definition}
\newtheorem{definition}[theorem]{Definition}

\theoremstyle{remark}

\usepackage{hyperref} 
\hypersetup{
    colorlinks=true,       
    linkcolor=blue,          
    citecolor=magenta,        
    filecolor=magenta,      
    urlcolor=cyan           
}  

\usepackage[msc-links,abbrev]{amsrefs}

\def\QSet{\mbox\textup{kern.24em \vrule width.03em height1.48ex depth-.051ex \kern-.26em Q}}
\def\S{{\mathbf S}}
\def\M{{{\mbox{\rm I\kern-.2em M}}}}
\def\PSet{\mbox{\rm I\kern-.22em P}}

\def\P{{\mathbf P}}

\def\T{{\mathbf T}}
\def\D{{\mathbf D}}

\def\J{{\mathbf J}}

\def\<{{\langle}}
\def\>{{\rangle}}

\def\size{\operatorname{size}}

\def\density{\operatorname{density}}

\title {Weighted bounds for variational Walsh-Fourier series}

\author{Yen Do}
\address{ Department of Mathematics, Yale University, New Haven CT 06511, USA}
\email {yenquang.do@yale.edu}
\thanks{Research supported in part by grant NSF-DMS-0635607002.}

\author{Michael Lacey}  
\address{ School of Mathematics, Georgia Institute of Technology, Atlanta GA 30332, USA}
\email {lacey@math.gatech.edu}
\thanks{Research supported in part by grant NSF-DMS 0968499 
	and  a grant from the Simons Foundation (\#229596 to Michael Lacey).}

\begin{document}
\begin{abstract}
For $1<p<\infty$ and a weight $w\in A_p$ and a function in $L^p([0,1], w)$ we show that variational sums with sufficiently large exponents of its Walsh--Fourier series are bounded in $L^p(w)$. This strengthens a result of Hunt--Young and is a weighted extension of a variation norm Carleson theorem of Oberlin--Seeger--Tao--Thiele--Wright. The proof uses phase plane analysis and a weighted extension of a variational inequality of L\'epingle.
\end{abstract}

\maketitle

\section{Introduction} \label{s.intro}
Let $f$ be a measurable function on $[0,1]$. The Walsh--Fourier series sum of $f$ given by
$$\sum_{k\ge 0} \langle f, W_k\rangle W_k(x) \ \ ,$$
 is a dyadic analogue of the Fourier series. We shall recall the definition of the Walsh system of functions $(W_k)_{k\ge 0}$ in Section~\ref{s.walsh}. It is standard that boundedness in $L^p$ of the maximal Walsh-Fourier sum
 $$Sf(x) := \sup_{n} |(S_n f)(x)|  \qquad   S_n f(x):=\sum_{0 \le k \le n} \langle f, W_k\rangle W_k(x) \ \ ,$$
leads to a.e.\thinspace convergence of the Walsh--Fourier series of functions in $L^p$. For $1<p<\infty$, this result holds, and is the Carleson theorem \cite{MR0199631} on the pointwise convergence of Fourier series. Also see Hunt~\cite{MR0238019}, for $1<p<2$, and Sj{\"o}lin \cite{MR0241885} for the Walsh case.

We are concerned with weighted estimates. For $1<p<\infty$ recall that a positive a.e.\thinspace  weight $w$ is in $A_p$ if the following bound holds uniformly over (dyadic) intervals:
$$ [w] _{A_p} := \sup_{I} \frac{1}{|I|}\int_I  w(x)dx \Biggl[\frac{1}{|I|}\int_I  w(x)^{-1/(p-1)} dx \Biggr]^{p-1} < \infty  \, .$$

In this paper we prove the following theorems. Below, $S_n f$ is assumed $0$ for $n<0$.

\begin{theorem}\label{t.maincor} Let $1<p<\infty$ and $w\in A_p$. Then there is an $ R= R (p,[w] _{A_p})<\infty$ such that 
	for all $r\in (R,\infty]$ we have
	\begin{equation}\label{e.maincor}\Bigl\|\sup_{M, N_0<\dots<N_M} \Bigl[
		\sum_{j=1}^M |S_{N_j} f -  S_{N_{j-1}}f|^r\Bigr]^{1/r}\Bigr\|_{L^p(w)} \le C \|f\|_{L^p(w)}  
\end{equation}
for some constant $C$ depending only on $w$, $p$, $r$. 
\end{theorem}
The simpler endpoint case $r=\infty$ of Theorem~\ref{t.maincor} is the Walsh-Fourier analogue of a theorem of Hunt and Young \cite{MR0338655} (cf. \cite{MR2115460} for extensions to more generalized settings). For $r<\infty$, the estimate \eqref{e.maincor} gives more quantitative information about the convergence rate of Walsh--Fourier series. 

Theorem~\ref{t.maincor} is a consequence of the following more general theorem:
\begin{theorem}\label{t.main} Let $1 < p < \infty $ and $w\in A_q$ for some $q\in [1,p)$. Then for $r \in (2q,\infty]$ such that $1/r < 1/q-1/p$, it holds that 
\begin{equation}\label{e.main}\|\sup_{M, N_0<\dots<N_M} (\sum_{j=1}^M |S_{N_j} f -  S_{N_{j-1}}f|^r)^{1/r}\|_{L^p(w)} \le C \|f\|_{L^p(w)}  
\end{equation}
for some constant $C$ depending only on $w$, $p$, $q$, $r$. 
\end{theorem}

To see how Theorem~\ref{t.main} implies Theorem~\ref{t.maincor}, take $1<p<\infty$ and $w\in A_p$. Note that the $A_p$ condition is an open condition, so  for some $ \epsilon>0$, there holds $w\in A_{p-\epsilon}$ (see for instance \cite{MR0293384}), and then apply Theorem~\ref{t.main} for $q=p-\epsilon$.

The Fourier  case of Theorem~\ref{t.main}, corresponding to $w\equiv 1 \in A_1$, is a theorem of Oberlin--Seeger--Tao-Thiele-Wright \cite{oberlin-et-al-walsh} (cf. \cite{oberlin-et-al}). Using this result, one can see that 
the conclusion of Theorem~\ref{t.maincor} must depend upon $ w \in A_p$.  Suppose that there is a fixed $ 0< r < \infty $ 	
and $ 1< p < \infty $,  for  which \eqref{e.maincor} holds for all $ w\in A_p$.  Using Rubio de Francia's extrapolation theorem, we see that this same 
inequality would have to hold for $ w$ being Lebesgue measure and \emph{all} $ 1 < p < \infty $.  This contradicts the (Fourier) 
examples that are in \cite{oberlin-et-al}*{Section 2}.  

The proof of Theorem~\ref{t.main} uses two main ingredients: adaptation of phase plane analysis to weighted settings, and a weighted extension of a classical variational inequality of L\'epingle (Lemma~\ref{l.weightLepingle}). The approach used in this paper is a weighted extension of the approach in \cites{oberlin-et-al, oberlin-et-al-walsh}, and in particular it is different from the elegant approach of Hunt--Young~\cite{MR0338655}, who use a good-$\lambda$ argument to upgrade the boundedness of the Carleson operator (the Fourier analogue of $S$) in the setting of Lebesgue measure to the settings of $A_p$ weights.  A naive adaptation of the 
good-$\lambda $ approach does not apply to the variational estimates for Carleson's operator. 

We became interested in new approaches towards boundedness of Walsh--Fourier series in weighted settings while investigating questions related to weighted bounds for multilinear oscillatory operators, such as the bilinear Hilbert transform (whose boundedness in the Lebesgue setting is well-known from the work of Lacey and Thiele \cites{MR1783613,MR1689336}).  To the knowledge of the authors, there hasn't been any adaptation of the Hunt--Young approach to the setting of  multilinear oscillatory operators.\footnote{We would like to point out that Xiaochun Li \cite{XiaochunLi} has some unpublished results about weighted estimates for the bilinear Hilbert transform.} Standard approaches towards multilinear oscillatory operators (started with Lacey--Thiele \cite{MR1783613} and further developed by Muscalu--Tao--Thiele \cites{MR2127984, MR1952931, MR2127985, MR2221256}) require detailed analysis on the phase plane, and this motivates us to consider a weighted adaptation of the time-frequency analysis framework. 

In this paper, we only consider analysis on the Walsh phase plane, which is certainly easier than the Fourier case, although there are qualitative similarities between the two phase planes. Extension of the argument in this paper to the Fourier setting is a nontrivial task. In the weighted setting, there is a lack of $L^2$ orthogonality for Walsh packets, therefore some changes are needed in the way one proves the so-called size lemma. In fact, we will use a sharp function estimate similar to an argument of Rubio de Francia in \cite{MR850681}, one can view this as a substitute for the good-$\lambda$ argument of Hunt--Young in the phase-plane. Our proof of Theorem~\ref{t.main} requires a weighted extension of  the L\'epingle inequality for variation norms and this is proved in Lemma~\ref{l.weightLepingle}.

\section{Walsh functions and Walsh packets}\label{s.walsh}

We recall standard properties of Walsh functions and Walsh packets below. A good reference is \cite{ThielePhDThesis}. The Walsh functions $W_0(x), W_1(x),\dotsc$ are supported in $[0,1]$ and can be defined recursively by $ W_0(x) = 1_{[0,1)}$, and 
for even and odd integers, 
\begin{gather*}
W_{2n}(x) = W_n(2x)1_{[0,\frac 12)} + W_n(2x-1)1_{[\frac 1 2, 1)} \qquad n\ge 1  
\\
W_{2n+1}(x) = W_n(2x)1_{[0,\frac 12)} - W_n(2x-1)1_{[\frac 1 2, 1)}  \qquad n\ge 0\,.
\end{gather*}

A dyadic rectangle in $ \mathbb R _+ \times \mathbb R _+$, with area one,  is referred to as a \emph{tile}. 
The Walsh packet associated with a tile 
$$p=[2^jm, 2^j(m+1))\times [2^{-j}n, 2^{-j}(n+1)) \  \equiv  I_p \times \omega_p$$ 
is an $L^2$ normalized function supported in the spatial interval $I_p :=[2^jm, 2^j(m+1))$ and is defined by
$$\phi_p(x) = 2^{-j/2} W_n(2^{-j}x-m) \,.$$

For two tiles $p_1$ and $p_2$ such that $p_1\cap p_2 \ne \emptyset$, we say that $p_1 < p_2$ if $I_{p_1} \subset I_{p_2}$. This clearly implies $\omega_{p_1}\supset \omega_{p_2}$, furthermore there is a close connection between the partial order 
and orthogonality.  Two tiles $ p_1$ and $ p_2$ are not ordered under `$ <$' if and only if the tiles do not intersect in the plane if and only if $ \langle \phi _{p_1}, \phi _{p_2} \rangle=0$.   

A dyadic rectangle of area 2 is referred to as a \emph{bitile}.  We will denote these as capital letters, like $ P$. 
There is an analog of the  partial order `$ <$' on tiles for bitiles, and we will use the same notation for it.   
A bitile $ P$ can be divided into two tiles having separate frequency intervals, a lower tile denoted by $P_1$  and an upper tile by $P_2$.
We say that $ P_1$ and $ P_2$ are \emph{siblings.}

The following property of Walsh packets is standard and has been used implicitly in various work on analysis of the Walsh phase plane (cf. \cites{MR2127984, MR1952931}). We formulate this property below and sketch a proof for the convenience of the reader, and since we will use it several times.

\begin{lemma}\label{l.walshtohaar} Suppose that two tiles $p$ and $p'$ are siblings and $q$ is another tile such that $p<q$. Let $I$ be the common time interval of $p$ and $p'$. Then there exist two constants $c_{p,q}$ and $c_{p',q}$ such that
$$\phi_p(x) = c_{p,q} 1_I(x) \phi_q(x) $$
$$\phi_{p'}(x) = c_{p',q} |I|^{1/2} h_I(x) \phi_q(x)$$
where $h_I$ is the Haar function associated with the dyadic interval $I$.
\end{lemma}
Note that one can easily compute the absolute values of $c_{p,q}$ and $c_{p',q}$:
$$|c_{p,q}| = |c_{p',q}| = \frac{|I_q|^{1/2}}{|I|^{1/2}} \,.$$
\begin{proof}
[Sketch of proof] For the first property, by induction one can assume $|I_q| = 2|I_p|$, in which case it follows from the recursive definition of $W_n$ (cf. \cite{ThielePhDThesis}). For the second property, note also that by definition $W_{2n+1}(x) = W_{2n}(x) h_{[0,1)}(x)$, so after appropriate scaling and modulation, it is clear that there is some constant $\alpha \in \{-1,1\}$ such that
$$\phi_{p'}(x) = \alpha |I|^{1/2} h_I(x) \phi_p(x) \,.$$
\end{proof}
  
\section{Discretization}

For any collection $\P$ of bitiles and any $r\in [1,\infty)$, let
$$C_{r,\P} f(x) := \sup_{M, N_0<\dots<N_M} \Big(\sum_{j=1}^M |\sum_{P \in \P} \langle f,\phi_{P_1}\rangle  \phi_{P_1}(x) 1_{\{N_{j-1} \not\in \omega_{P}, \ N_j \in \omega_{P_2}\}}|^r\Big)^{1/r} \,.$$
A symmetric variant of $C_{r,\P}$ can be obtained by using the limiting conditions $\{N_{j-1}\in \omega_{P_1}, N_j \not\in \omega_P\}$ in the above expression.

In the rest of the paper we'll always assume that $r<\infty$ and $r>2q$, and $q>1$. These assumptions are without loss of generality.

Via a standard argument (cf. \cite{ThielePhDThesis}), Theorem~\ref{t.main} follows from the following theorem and its symmetric variant (whose proof is completely analogous).
\begin{theorem}\label{t.discretizedmain} There is a constant  $ C = C(w, p,  q r)>0$ such that for any collection $\P$ of bitiles we have
\begin{equation}\label{e.discretizedmain}
\|C_{r,\P}f\|_{L^p(w)} \le C \|f\|_{L^p(w)} \,
\end{equation}
for all $p \in (q,\infty)$  such that $1/r>1/q-1/p$.
\end{theorem}
By duality  (cf. \cite{oberlin-et-al}), it suffices to show \eqref{e.discretizedmain} for the following linearized variant of $C_{r,\P}$ (we'll omit the dependence on $r$ for simplicity):
$$(C_{\P}f)(x)  = \sum_{j=1}^{M(x)}\sum_{P \in \P} \langle f,\phi_{P_1}\rangle  \phi_{P_1}(x) 1_{\{N_{j-1}(x) \not\in \omega_{P}, \ N_j(x) \in \omega_{P_2}\}}a_j(x) \ \ ,$$
$$\sum_{j=1}^{M(x)} |a_j(x)|^{r'} =  1 \ \ .$$
In the following, we denote $a_P(x) = \sum_{j=1}^{M(x)} a_j(x)1_{\{N_{j-1}(x) \not\in \omega_{P}, \ N_j(x) \in \omega_{P_2}\}}$.  Let
$$B_{\P}(f,g) :=  \sum_{P\in \P} \langle f,\phi_{P_1}\rangle  \langle \phi_{P_1} \, a_P \, , \, gw\rangle  \,.$$ 
Also, denote $ w(G) := \int_G w(x) dx$ for any set $ G$. We say that $K \subset G$ is a major subset of $G$ if $w(K) > w(G)/2$ and we say it has full measure if $w(K)=w(G)$. We'll show that
\begin{proposition}\label{p.restrictedweaktype} Let $F$ and $G$ be two sets with $w(F)$, $w(G)<\infty$. Then there exists $\widetilde{F}$ and $\widetilde{G}$,  major subsets of $F$ and $G$ respectively such that:\\
(i) at least one of them has full measure, and\\
(ii) for any $ \lvert  f\rvert  \le 1_{\widetilde F}$ and  $|g| \le 1_{\widetilde G}$ and any collection of bitiles $\P$ we have
\begin{equation}\label{e.restrictedweaktype}
B_{\P}(f, g) \le C w(F)^{1/p} w(G)^{1-1/p} \, \, 
\end{equation}
for all $p \in (q,\infty)$  such that $1/r>1/q-1/p$.
\end{proposition}

Via restricted weak-type interpolation, the above proposition implies
$$B_{\P}(f,g) \le C \|f\|_{L^p(w)} \|g\|_{L^{p'}(w)} \, \qquad   p > q $$
where $C$ depends only on $p$,$q$,$r$,$w$ and there is no restriction on $f$ or $g$, and this in turn implies Theorem~\ref{t.discretizedmain}.  It remains to show Proposition~\ref{p.restrictedweaktype}.

\section{Decomposition of {\bf P}}
To prove Proposition~\ref{p.restrictedweaktype}, as is now standard, $\P$ will be decomposed into more refined subcollections, so that the bilinear sum $B$ associated with each such subcollection can be estimated more effectively. For this purpose, two standard measurements, \emph{size} and \emph{density}, are associated with each  collection. In this section we formulate our weighted adaptations of these notions.

To formulate size, we first recall the definition of trees.
\begin{definition}[Tree]Let $P_T$ be a bitile. A tree $T$ with tree top $P_T$ is a finite collection of bitiles such that $P < P_T$ for any $P\in T$. 
\end{definition}
Writing $P_T = I_T \times \omega_T$, we will refer to $I_T$ as the top interval of $T$. A tree is called $1$-overlapping if the lower tile $P_1$ of every $P\in T$ is less than the lower tile of the tree top. Similarly, a tree is $2$-overlapping if every upper tile $P_2$ is less than the upper tile of the tree top. Clearly any tree can be decomposed into two trees, one of each type.

In the following, let
$S_Tf(x)  := \Bigl[\sum_{P\in T} |\langle f,\phi_{ P_1}\rangle |^2\frac{1_{I_P}}{|I_P|} \Bigr] ^{1/2}  $.

\begin{definition}[Size] 
The size of a collection of bitiles $\P$ is the best constant $C$ such that: for any $2$-overlapping tree $T \subset \P $ we have
$$\Bigl\|S_T f\Bigr\|_{L^{2}(w)} \le C w(I_T)^{\frac 1{2}} \ \ .$$
We will denote the size of $\P$ by $\size(\P)$. 
\end{definition}
It is clear that for $w\equiv 1$ one recovers the standard definition of size (cf. \cites{MR1689336,MR2127984}).

\begin{definition}[Density]
The density of a collection $\P$ of bitiles is
$$\density(\P):= \sup_{P\in \P} \sup_{Q>P} \Big(\frac{1}{w(I_{Q})} \int_{I_{Q}} |g(x)|^{r'}\sum_{k: N_k(x)\in \omega_{Q}} |a_k(x)|^{r'}w(x)dx \Big)^{1/r'}\,.$$
\end{definition}
Since $|g|\le 1_G$. it is clear that the density of any collection is bounded above by $1$.

\subsection{Size bounds} In this section, we show the following bound, which is a variant of \cite{MR2127984}*{Lemma 4.5}.
\begin{lemma}\label{l.size-bound}  If $w$ is in $A_q$ then
\begin{equation}\label{e.size-bound}
\size(\P) \le C \sup_{P\in \P} \Big(\frac{w(I_P \cap F)}{w(I_P)}\Big)^{1/q} \,.
\end{equation}
\end{lemma}
The proof of Lemma~\ref{l.size-bound} relies on the following BMO characterization of size, which is a variant of \cite{MR2127984}*{Lemma 4.2}. 
\begin{lemma}\label{l.BMOcharacterize}
For any collection $\P$ of bitiles and any $1<p<\infty$ we have
\begin{equation}\label{e.BMOcharacterize}
\sup_{T \subset \P} \frac{1}{w(I_T)^{1/p}} \|S_Tf\|_{L^{p}(w)}  \sim_p \sup_{T \subset \P} \frac{1}{w(I_T)} \|S_Tf\|_{L^{1,\infty}(w)} 
\end{equation}
the suprema are over $2$-overlapping trees.
\end{lemma}

\begin{proof} Since $S_T f$ is supported in $I_T$, the right hand side in \eqref{e.BMOcharacterize} is clearly bounded above by the left hand side. For the other direction, one can freely assume that $\P$ is finite. Denote the left hand side of \eqref{e.BMOcharacterize} by $\sigma $, which is now finite.     

Let $T\subset \P$ be a $2$-overlapping tree such that
\begin{equation}\label{e.BMOlargesize}
\|S_Tf\|_{L^{p}(w)} \ge \frac{\sigma }{2}{w(I_T)^{\frac 1{p}}} \,. 
\end{equation}
We will show that the $ L ^{1, \infty } (w)$ norm of $ S_T f$,  tested at height $ \lambda \simeq \sigma $, dominates $w(I_T) \sigma $.  
For any dyadic interval $I$, by definition of $\sigma $ we have
$$\Bigl\|\Big(\sum_{P\in T, I_P \subset I} |\langle f,\phi_{ P_1} \rangle|^2\frac{1_{I_P}}{|I_P|} \Big)^{p/2}\Bigr\|_{L^1(w)}  \le \sigma ^{p} w(I) \,.$$
Note that the integrand on the left hand side is supported in $I$. Now, fix  $\lambda >0$ and let 
$$\widetilde T: = \{P \in T: I_P \subset  \{S_Tf > \lambda\} \} \,.$$
By dividing $\{S_Tf > \lambda\}$ into maximal dyadic components and applying the last estimate for each such interval, after summing we obtain

\begin{equation}\label{e.BMOlocal}
\|(S_{\widetilde T}f)^{p}\|_{L^1(w)} \le \sigma^{p} w(\{S_Tf > \lambda\})  
\end{equation}
On the other hand, it is not hard to see that $\|S_{ T\setminus \widetilde T}f\|_\infty \le \lambda $. Indeed, one only needs to show that for any maximal dyadic component $I$ of $\{S_Tf > \lambda\}$ and any $x\in I$ we have
$$S_{T\setminus \widetilde T}f(x) \le \lambda \ \ .$$
Let $J$ be the dyadic parent of $I$. By definition of $\widetilde T$, one can write
$$S_{T\setminus \widetilde T}f(x) = S_{T_J} f(x) \ \ , \ \ T_J := \{P \in T: J\subset I_P\} \ \ .$$
Clearly $S_{T_J}f$ is constant on $J$ and $J$ has nontrivial intersection with $\{S_Tf \le \lambda\}$. Therefore
$$S_{T_J}f(x) = \inf_{y\in J} S_{T_J}f(y) \le \inf_{y\in J} S_Tf(y) \le \lambda \ \ .$$

Since $S_Tf$ is supported inside $I_T$, using \eqref{e.BMOlargesize} and \eqref{e.BMOlocal} and the above $L^\infty$ bound, we have

$$\frac{\sigma ^{p}w(I_T)}{2^p} \le \|(S_Tf)^{p}\|_{L^1(w)} \le C \lambda^{p} w(I_T) + C\sigma ^{p} w(\{S_Tf > \lambda\}) \,.$$
Letting $\lambda =  \sigma /C$ for some large $C$, we obtain the desired estimate: for some $c>0$,
$$c \sigma  w(I_T) \le \lambda w(\{S_Tf > \lambda\}) \le \|S_Tf\|_{L^{1,\infty}(w)} \,.$$
\end{proof}

\begin{proof}[Proof of Lemma~\ref{l.size-bound} using Lemma~\ref{l.BMOcharacterize}] 

By Lemma~\ref{l.BMOcharacterize}, it suffices to show
$$\|S_T(f)\|_{L^q(w)} \le C w(I_T)^{1/q}\sup_{P\in T} \Big(\frac{w(I_P \cap F)}{w(I_P)}\Big)^{1/q}$$
for each $2$-overlapping tree $T$.  One can assume that $T$ contains its top element, in which case we will  show:
$$\|S_T(f)\|_{L^q(w)} \le C w(I_T \cap F)^{1/q} \ \ .$$
Let $(\epsilon_P)_{P\in T}$ be a random sequence of $1$ and $-1$, then it suffices to show the following uniform estimate (over $\epsilon$):
$$\|T_\epsilon f\|_{L^q(w)} \le C \|f\|_{L^q(w)}$$
$$T_\epsilon f:=\sum_{P\in T} \epsilon_P \langle f,\phi_{ P_1} \rangle\phi_{ P_1} \,.$$ 
By the $2$-overlapping property of $T$ , by Lemma~\ref{l.walshtohaar}, we can rewrite $T_\epsilon$ as
$$(T_\epsilon f)(x) =  {|I_T|} \sum_{P\in T} \epsilon_P \langle f\phi_{P_T}, h_{I_P}\rangle  h_{I_P}(x)\phi_{P_T}(x)$$
{ where $ \phi _{P_T} $ is the Walsh packet associated to the upper tile of the top of the tree.} Therefore the desired bound for $T_\epsilon$ follows from standard properties of the martingale transform (cf. \cite{MR1748283}).
\end{proof}

\subsection{Tree selection by size}
The decomposition of the collection $\P$ is done via selection of trees of comparable size and density. The following Lemma allows for selection of trees based on size. Recall that $1< q < \infty $,  and $ w \in A_q$.  
\begin{lemma}\label{l.treeselect-bysize} 
Let $\P$ be a collection of bitiles  with $ \sigma = \size (\P)< \infty $. 
Then there exists a subcollection $\P' \subset \P $ with $$\size(\P')<\sigma/2$$ 
such that $\P \setminus \P'$ can be written as a union of trees, $ \P \setminus \P' = \bigcup_{T\in \T} T$, with
$$\sum_{T\in \T} w(I_T) \le C \sigma^{-2q} w(F) \, \, . $$
The constant $ C$ depends upon $ q$ and  $[w] _{A_q} $.  
\end{lemma}

This proof, especially the appeal to the sharp function below,  is much easier to complete in the Walsh setting. { We note that the usual approach (cf. \cite{MR1783613}) relies on some orthogonality of the packets in $L^2$, and this is not necessarily true for non-Lebesgue weights $w$.}
Our proof strategy for Lemma~\ref{l.treeselect-bysize} is derived from Rubio de Francia's argument \cite{MR850681}.

\begin{proof} 
By the standard selection algorithm (cf. \cite{MR1783613} or \cite{MR2127984} which is closer to the dyadic setting of this paper), one can find a collection of trees  $\T$ such that the following conditions hold.  Each $T \in \T$ contains a $ 2$-overlapping tree $ T_2$ such that 
\begin{gather*} 
\|S_{T_2}f\|_{L^{2}(w)} \ge \frac \sigma 2 \omega(I_T)^{\frac 1{2}}   \qquad    T \in  \T  
\\
\size(\P \setminus \bigcup_{T\in \T} T) < \sigma/2 \,. 
\end{gather*}
Furthermore, the selection algorithm ensures that the tiles in the collection  
$\D:=\{P_1: P\in \bigcup_{T\in \T} T_2\}$ are pairwise disjoint tiles in the phase plane.

It remains to bound the sum over $T\in \T$ of $w(I_T)$'s. Using H\"older's inequality, we have 
$$\|S_{T_2}f\|_{L^{2q}(w)} \ge \frac \sigma 2 \omega(I_T)^{\frac 1{2q}}$$
therefore 
\begin{align*}
\sum_{T\in \T} w(I_T) &\le C\sigma^{-2q} \Bigl\|\sum_{T\in \T} (S_{T_2}f)^{2q}\Bigr\|_{L^1(w)}
\\
&\le C\sigma^{-2q} \Bigl\|\Big(\sum_{T\in \T} (S_{T_2}f)^{2}\Big)^{q}\Bigr\|_{L^1(w)}
\\
&= C\sigma^{-2q} \Bigl\|\Big(\sum_{p\in \D} |\langle f,\phi_p\rangle |^2\frac{1_{I_p}}{|I_p|}\Big)^{1/2}\Bigr\|_{L^{2q}(w)}^{2q} \,.
\end{align*}
Let $S_{\D}f$ denote the square sum inside the  last $L^{2q}$ norm. We will  show
\begin{equation}\label{e.sharpM2}
(S_{\D}f)^{\sharp} \le C M_2f  
\end{equation}
where the left hand side is the dyadic sharp maximal function of $S_{\D}f$, and
$$M_2f(x) := \sup_{I: x\in I} \Bigl(\frac{1}{|I|} \int_I |f(x)|^2dx\Bigr)^{1/2} \,.$$
Since $q>1$ and $w\in A_q $, \eqref{e.sharpM2} implies the desired estimate:
\begin{align*}
\lVert S _{\D} f\rVert_{L^{2q}(w)}^{2q} &\le C _{q,w}  
\|(S_{\D}f) ^{\sharp}\|_{L^{2q}(w)}^{2q}  
\\
&\le C\|M_2f\|_{L^{2q}(w)}^{2q} \le C\|f\|_{L^{2q}(w)}^{2q} \le C w(F) \,.
\end{align*}
Note that we are appealing to $ \lVert g\rVert_{ L ^{2q} (w)} \le C _{q,w} \lVert g ^{\sharp}\rVert_{ L ^{2q} (w)} $, and  
 in the second inequality we used boundedness of the maximal function on $L^{q}(w)$. It remains to show \eqref{e.sharpM2}. 
 
 Take a dyadic interval $J$ and  $x \in J$. In the definition of the sharp maximal function, we are permitted to subtract off a constant, and we will take that constant to be 
$$c_J = \Big(\sum_{p\in \D: J\subset I_p} \frac{|\langle f,\phi_p\rangle |^2}{|I_p|}\Big)^{1/2} \,.$$
Then via H\"older's inequality, we have
\begin{align*}
\Biggl[\frac 1{|J|}\int_{J} |S_{\D}(f) - c_J| dy \Biggr] ^2  &\le  \frac 1{|J|}\int_{J} |S_{\D}(f)^2 - c_J^2| dy \\
&\le \frac 1{|J|}\sum_{p\in \D} |\langle f1_J, \phi_p\rangle |^2 
\\
&\le \frac 1{|J|}\int_J |f(y)|^2 dy  \le  M \lvert f\rvert ^2 (x) \,.
\end{align*}
This proves \eqref{e.sharpM2}.
\end{proof} 

We shall also need the following result (cf. \cite[Proposition 4.3]{oberlin-et-al}).
\begin{lemma} \label{l.tree-redecomp} The collection of trees selected in Lemma~\ref{l.treeselect-bysize} also satisfies for any $p\in (1,\infty)$:
\begin{equation}\label{e.BMO}
	\Bigl\|\sum_{T\in \T} 1_{I_T}\Bigr\|_{L^p(w)} \le C\sigma^{-2q} w(F)^{\frac 1p} \  ,
\end{equation}
and if $\P = \bigcup_{S\in \S} S$ is another tree decomposition of $\P$ then
\begin{equation}\label{e.efficient}
\sum_{T\in \T} w(I_T) \le C \sum_{S\in \S} w(I_S) \ \ .
\end{equation}
\end{lemma}
The last condition quantifies an efficient aspect of the tree selection algorithm. 

\proof We first prove \eqref{e.BMO}. Let $M_w$ be the weighted maximal function. 
$$M_w f(x) = \sup_{I: x\in I} \frac{1}{w(I)}\int_I |f(y)|w(dy) \ \ .$$ 
Let $N = \sum_{T\in \T} 1_{I_T}$, then it suffices to show the good $\lambda$ inequality
\begin{equation}\label{e.goodlambda}
w(\{N > \lambda, M_w 1_F \le c\sigma^{2q}\lambda\}) \le \frac 1{1000} w(\{N>\lambda/2\}) 
\end{equation}
for some small absolute constant $c>0$. Indeed, it follows from \eqref{e.goodlambda} that
$$\|N\|_{L^p(w)}^p = \int_0^\infty p\lambda^{p-1} w(\{N > \lambda\}) d\lambda $$
$$\le C \int_0^\infty p\lambda^{p-1} w(\{M_w 1_F > c\sigma^{2q}\lambda\}) d\lambda$$
$$= C\sigma^{-2pq}\|M_w 1_F\|_{L^p(w)}^p \le C \sigma^{-2pq}\|1_F\|_{L^p(w)}^p = C \sigma^{-2pq}w(F)\ \ .$$

To prove \eqref{e.goodlambda}, decompose $\{N>\lambda/2\}$ into maximal dyadic intervals, and it suffices to show that for any such maximal $I$ with nontrivial intersection with $\{M_w 1_F \le c\sigma^{2q}\lambda\}$ we have
$$w(\{x \in I: N(x) > \lambda \}) \le \frac 1{1000} w(I) \ \ .$$
Let $\T_I = \{T\in \T: I_T\subset I\}$. Then the argument in Lemma~\ref{l.treeselect-bysize} applied to $f1_I$ gives
\begin{align*}
\sum_{T \in \T_I} w(I_T) \le C \sigma^{-2q} w(F\cap I)
& \le C w(I) \sigma^{-2q}\inf_{x\in I} (M_w 1_F)(x)
\\&\le Cw(I)\sigma^{-2q} c^{2q}\sigma^{2q} \lambda
 \le \lambda w(I)/4000 
\end{align*} if $c$ is chosen sufficiently small.

Consequently, for $N_I:= \sum_{T\in \T_I} 1_{I_T}$ we have
$$w(\{N_I > \lambda/4\}) \le 4\lambda^{-1}\|N_I\|_{L^1(w)} \le |I|/1000 \ \ .$$
Now, $N -N_I$ is constant on the parent $\pi(I)$ of $I$, is dominated by $\inf_{x\in\pi(I)} N(x)$, which in turn is less than $\lambda/2$ by maximality of $I$. Thus
$$\{x\in I: N(x) > \lambda\} \subset \{x\in I: N_I(x) > \lambda/4\}$$
and \eqref{e.goodlambda} follows.
 
\smallskip 

Now we'll show \eqref{e.efficient}. By the selection algorithm, we have 
\begin{align*}
\sum_{T\in \T} w(I_T) &\le C \sigma^{-2} \sum_{T\in \T} \int \sum_{P\in T_2}|\<f, \phi_{P_1}\>|^2 \frac{1_{I_P}}{|I_P|} w(x)dx
\\&= C \sigma^{-2} \sum_{P\in D} |\<f, \phi_{P_1}\>|^2 \frac{w(I_P)}{|I_P|} \ \ ,
\\
\textup{where} \quad D & :=\bigcup_{T\in \T} T_2 \ \ .
\end{align*}
We'll show that
$$\sum_{P\in D} |\<f, \phi_{P_1}\>|^2 \frac{w(I_P)}{|I_P|} \le C\sigma^2 \sum_{S \in \S} w(I_S)$$
and that will complete the proof of \eqref{e.efficient}.

Now, for any tree $S\in \S$ we can decompose $S\cap D$ into two trees $S_1$ and $S_2$ with the same top interval, where $S_1$ is $1$-overlapping and $S_2$ is $2$-overlapping. Clearly $\bigcup_{S\in \S} (S_1\cup S_2) = D$. By given assumption, we have
$$\sigma \ge \size(S_2) \sim \frac{1}{w(I_S)^{1/2}} (\int \sum_{P\in S_2} |\<f, \phi_{P_1}\>|^2 \frac{1_{I_P}}{|I_P|} w(x)dx)^{1/2}$$
therefore
\begin{equation}\label{e.S2}
\sum_{P\in S_2} |\<f, \phi_{P_1}\>|^2 \frac{w(I_P)}{|I_P|} \le C\sigma^2 w(I_S) \ \ .
\end{equation}
On the other hand, the selection algorithm ensures that the $1$-tile of any two elements of $D$ are disjoint. Therefore each $S_1$ contains only spatially disjoint elements. If $P\in S_1$ then
$$\frac{1}{w(I_P)} |\<f,\phi_{P_1}\>|^2 \frac{w(I_P)}{|I_P|} \le [\size(\{P\})]^2 \le \sigma^2$$
so we obtain
\begin{equation}\label{e.S1}
\sum_{P\in S_1} |\<f, \phi_{P_1}\>|^2 \frac{w(I_P)}{|I_P|} \le C\sigma^2 \sum_{P\in S_1} w(I_P) \le C \sigma^2 w(I_S) \ \ .
\end{equation}
Summing over $S\in \S$ of \eqref{e.S1} and \eqref{e.S2} we obtain the desired estimate.\endproof

\subsection{Tree selection by density}
The proof of the next Lemma follows from standard arguments, we omit details (cf. \cite{MR1783613}). 
\begin{lemma}\label{l.treeselect-bydensity}
	Let $\lambda >0$ and let $\P$ be a collection of bitiles. Then there is an $ \P' \subset \P$ with $\size(\P')<\lambda/2$ 
such that $\P \setminus \P'$ can be written as a union of trees $\P \setminus \P= \bigcup _{T\in \T} T$ with
$$\sum_{T\in \T} w(I_T) \le C \lambda^{-r'} w(G)\,.$$
\end{lemma}

\section{The tree estimate}
The estimates of the bilinear sums $B _{\P}$  are based on the following estimate:

\begin{lemma}\label{l.tree-est}
Let $T$ be a tree, then for any $s\in [1,r']$ we have
\begin{equation}\label{e.tree-est}
\|gC_{\T} f \|_{L^{s}(w)}  \le C w(T)^{1/s} \size(T) \density(T) \,.
\end{equation}
\end{lemma}

\begin{proof} By H\"older's inequality it suffices to show \eqref{e.tree-est} for $s=r'$. By dividing $ T$ into two subtrees,  if necessary,   we can assume that the tree is either $1$-overlapping or $2$-overlapping. We will return to this dichotomy below.  

Let $\J$ be the set of maximal dyadic intervals inside $I_T$ that does not contain any $I_P$ for $P\in T$. This collection partitions $I_T$, and we rewrite the left hand side of \eqref{e.tree-est} as
$$\Bigl(\sum_{J \in \J} \int_J |(C_{T}f)(x)  g(x)|^{r'} w(x) dx\Bigr)^{1/r'} \ \ .$$ 
Fix $J\in \J$. By maximality of $J$,  there is some $P_J \in T$ such that $I_{P_J} \subset \pi(I)$, where $\pi(J)$ is the dyadic parent of $J$.  It is clear that there is  a bitile $Q_J$ such that
$$P_J < Q_J  \ \ I_{Q_J} = \pi(J)  \ \ .$$
In particular, $w_{Q_J} \cap w_{P_T} \ne\emptyset$. On the other hand, again by maximality of $J$, for any $P\in T$ such that $I_P\cap J\ne \emptyset$ we have $J \subsetneqq  I_P$. 
Consequently, $w_{P} \subset w_{Q_J}$ for those $P$'s, thus
\begin{align}
\label{e.localdensitybound1}
\bigcup_{P\in T: \, I_P \cap J \ne \emptyset} \omega_{P_2} \subset \omega_{Q_J} \ \ .
\end{align}
Furthermore, it is clear that
\begin{align}
\label{e.localdensitybound2} \int_J \sum_{k: N_k(x) \in \omega_{Q_J}} |a_k(x)g(x)|^{r'} w(x) dx  \le C w(J) \density(T)^{r'} \,. 
\end{align}
Here, the constant $ C$ depends upon the  doubling property of $ w$, which is controlled by $ [w ] _{A_q}$. 

\noindent \underline{Case 1: $T$ is $1$-overlapping}. Then the tiles  $\{P_2: P\in T\}$ are disjoint. Then by monotonicity of $N_k$'s, 
for any $ x$ there is at most one $P\in T$  such that  there is a $ k\in [1, M(x)]$ satisfying both $(x,N_k(x)) \in P_2$ and $(x,N_{k-1}(x)) \not\in P$. Clearly, such $k$ if exists is unique. Consequently, using \eqref{e.localdensitybound1} and \eqref{e.localdensitybound2} we have
\begin{align*}
	\Big(\sum_{J \in \J} \int_J |(C_{T}f)(x) &	 g(x)|^{r'} w(x) dx\Big)^{1/r'}
\\
&\le C \sup_{P \in T}  \frac{|\langle f,\phi_{P_1}\rangle |}{|I_P|^{1/2}}  \Big(\sum_{J \in \J} \int_J  \sup_{k} |a_k(x) 1_{\{N_k(x) \in \omega_{Q_J} \}}  g(x)|^{r} w(x)dx \Big)^{1/r'}
\\
&\le C \sup_{P\in T} \frac{|\langle f,\phi_{P_1}\rangle |}{|I_P|^{1/2}} \density(T) \Big(\sum_{J \in \J}w(J)\Big)^{1/r'}
\\
&= C \sup_{P\in T} \Big(\frac{1}{w(I_P)} \int \Big[|\langle f,\phi_{P_1}\rangle |^2 \frac{1_{I_P}}{|I_P|}\Big] w(x) dx\Big)^{\frac 1{2}} \density(T) w(I_T)^{1/r'}
\\
&\le C w(I_T)^{1/r'}\size(T) \density(T) \, \,.
\end{align*}

 \noindent \underline{Case 2: $T$ is $2$-overlapping}. From Lemma~\ref{l.walshtohaar}, it follows that we can write 
\begin{equation*}
\sum_{P\in T} \langle f, \phi _{P_1} \rangle  \phi _{P_1} 
= \phi _{P_T} \lvert  I _{T}\rvert ^{1/2}   \sum_{P\in T}   {\widetilde\epsilon_P} \langle f,  \phi _{P_1} \rangle    h _{I_P}  
\end{equation*}
here $\widetilde \epsilon_P = \pm \epsilon_P$ and the sign depends on the sign of the implicit constant in the application of Lemma~\ref{l.walshtohaar}.
Also, $ \phi _{P_T} $ is the Walsh packet associated to the  \emph{upper} tile of the top of the tree, 
so that $ \lVert  \phi _{P_T} \lvert  I _{T}\rvert ^{1/2} \rVert_{ \infty } =1$. In particular, we can ignore this term in the 
considerations below.  For convenience, below we denote $\varphi_T = \sum_{P\in T}   {\widetilde\epsilon_P} \langle f,  \phi _{P_1} \rangle    h _{I_P}$.

Now, for convenience denote by $\Delta_j$ the projection of a function onto the space generated by Haar functions adapted to dyadic intervals of length $2^{1-j}$. The function $ \varphi_T $, being a linear combination of Haar functions, satisfies 
the familiar identity below, for any dyadic interval $ K$:  
\begin{equation*}
\sum_{\substack{P\in T\\ K\subsetneq I_P}}    { \widetilde \epsilon _{P}} \langle f , \phi _{P_1}  \rangle    h _{I_P} 
= \Delta_{-\log_2|K|} (\varphi_T) \ \ \text{, and is locally constant on $K$}.
\end{equation*}
Now, since $T$ is a tree, the intervals $\omega_P$ for $P\in T$ are clearly nested. Furthermore, the $2$-overlapping property of $T$ means that the intervals $ \omega _{P_2}$ 
for $ P\in T$ are also nested.  Hence, if $ N \in \omega _{P_2}$, then $ N \in \omega _{P'_2}  $ for all tiles $ P' \in T$ with $ \lvert  I_{P'}\rvert \le \lvert  I _{P}\rvert  $, and if $N \not\in \omega_P$ then $N\not\in\omega_{P'}$ for all $P'\in T$ with $|I_{P'}| \ge |I_P|$.  Combining these observations, for any $J\in \mathbf J$ we can find measurable functions defined on $J$
$$\tau_M(x) < \xi_M(x) \le \dots \le \tau _{1} (x) < \xi_1(x) \le \tau_0(x)< \xi_0(x) \le -\log_2|J|$$
such that for any $x\in J$:
\begin{align*}
\sum_{k=1}^{M(x)}& \Biggl\lvert \sum_{P\in T: J\subsetneq I_P} { \widetilde \epsilon _{P}} \langle f, \phi _{P_1} \rangle    h _{I_P} (x)a_k(x)1_{\{N_{k-1}(x) \not\in \omega_P, \, N_k(x) \in \omega_{P_2}\}}  \Biggr\rvert
\\	&= 
\sum_{k}  \Bigl\lvert \sum_{\tau _k (x) < j \le \xi_{k} (x)} (\Delta_j \varphi_T )(x) \Bigr\rvert |a_k(x)|
\\
&\le \Big(\sum_{k}  \lvert \sum_{\tau _k (x) < j \le \xi_{k} (x)} (\Delta_j \varphi_T )(x) \rvert^{r}\Big)^{1/r} \Big(\sum_{k}|a_k(x)|^{r'}\Big)^{1/r'} 
\\
&= \Big [\frac 1{|J|} \int_{J} \big(\sum_{k}  \lvert \sum_{\tau _k (x) < j \le \xi_{k} (x)} (\Delta_j \varphi_T )(y) \rvert^{r}\big)^{1/r} dy\Big] \Big [\sum_{k}|a_k(x)|^{r'}\Big]^{1/r'} 
\\
&\le M_J (\|\varphi_T\|_{V^r}) \big(\sum_{k}|a_k(x)|^{r'}\big)^{1/r'} \ \ ;
\\ \textup{where} \quad 
\|\varphi_T\|_{V^r} &:= \sup_{m, n_0<\dots <n_m} \Big(\sum_{1\le k \le  m} |\sum_{n_{k-1}< \le n_k} (\Delta_j \varphi_T)|^r\Big)^{1/r} \ \ .
\end{align*}

Thus, using \eqref{e.localdensitybound1} and \eqref{e.localdensitybound2}, we obtain
\begin{align*}
\|C_T(f)g\|_{L^{r'}(w)} &\le C\Big(\sum_{J\in \J} M_J(\varphi_T)^{r'} w(J) \density(T)^{r'}\Big)^{1/r'}
\\
&\le C \|1_{I_T} M(\varphi_T)\|_{L^{r'}(w)} \density(T)
\\
&\le C w(I_T)^{\frac1{r'}-\frac 1{2q}}\|M(\|\varphi_T\|_{V^r})\|_{L^{2q}(w)} \density(T) \ \ \text{(using $2q > 2 > r'$)}
\\
&\le C w(I_T)^{\frac1{r'}-\frac 1{2q}} \|\|\varphi_T\|_{V^r}\|_{L^{2q}(w)} \density(T) \ \ \text{(using $w\in A_q \subset A_{2q}$)}
\\
&\le C w(I_T)^{\frac1{r'}-\frac 1{2q}}  \|\varphi_T\|_{L^{2q}(w)} \density(T) 
 \ \ .
\end{align*}
The last inequality depends upon the weighted variant of an inequality of L\'epingle taken up in the next section. 

Now, by a standard duality argument and boundedness of the dyadic square function on $L^{2q}(w)$ (cf. \cite{MR1748283}) one has
$$\|\varphi_T\|_{L^{2q}(w)}  \le C  \|S(\varphi_T )\|_{L^{2q}(w)}  $$
where $S(g) := (\sum_I |\langle g,h_I\rangle |^2 \frac 1{|I|})^{1/2}$, and the constant above depends upon $ q$ and $ w$.  
Using Lemma~\ref{l.walshtohaar} and the fact that $T$ is a tree, we obtain $ S (\varphi_T ) = S_Tf $. Lemma~\ref{l.tree-est} now follows, using the BMO characterization of size proved in Lemma~\ref{l.BMOcharacterize}. 
\end{proof}

\section{A weighted L\'epingle inequality}
For each $i$ let $\Delta_i$ be the projection onto the space of Haar functions adapted to dyadic intervals of length $2^{1-i}$:
$$\Delta_i f = \sum_{I: |I|=2^{1-i}} \<f, h_I\> h_I \ \ .$$
In this section we prove the following extension of an inequality of L\'epingle \cite{MR0420837} (cf. \cites{MR1019960, MR2434308, MR933985}).
\begin{lemma}\label{l.weightLepingle} Let $1<p<\infty$, $w\in A_p$ and $r>2$. Then for any function $f$ we have
\begin{equation}\label{e.rnot2}
\|\sup_{M, N_0<N_1<\dots<N_M} \Big(\sum_{k=1}^M |\sum_{N_{k-1}< j \le N_k} \Delta_j f|^r\Big)^{1/r}\|_{L^p(w)} \le C \|f\|_{L^p(w)} \ \ .
\end{equation}
Furthermore, the following endpoint estimate holds uniformly over $\lambda >0$:
\begin{equation}\label{e.r2}
\|\lambda M_\lambda^{1/2} \|_{L^p(w)} \le C\|f\|_{L^p(w)} \ \ ,
\end{equation}
$$M_\lambda(x) := \sup_{M, N_0<N_1<\dots<N_M} \sharp \{k: |\sum_{N_{k-1}< j \le N_k} \Delta_j f| > \lambda\}  \ \ . $$
\end{lemma}

The considerations in the proof are of a standard nature. 
\proof We first show that \eqref{e.r2} implies \eqref{e.rnot2} using an argument in \cite{do-muscalu-thiele} (cf. \cite{MR1019960}). By standard arguments, we can remove the supremum in the estimates and assume instead that $M, N_0<\dots < N_M$ are measurable functions of $x$. It suffices to show that if $w\in A_p$ then
$$w(\{V_r f > \lambda\}) \le C\lambda^{-p}\|f\|_{L^p(w)}^p  \ \ ,$$
from this the desired strong bound follows from interpolation (exploiting the reverse H\"older property and the nesting property of $A_p$ classes). Via scaling invariant, one can assume $\|f\|_{L^p(w)}=1$, and let $a_k$ denote $\sum_{N_{k-1}< j \le N_k} \Delta_j f$. Then on the set
$$E = \{x: \sup_{k} |a_k(x)| > \lambda\}$$
one has $M_\lambda(x) = \sharp\{k: |a_k|>\lambda\} \ge 1$, thus using \eqref{e.r2} one has
\begin{equation}\label{e.exception}
w(E) \le \int  M_\lambda(x)^{p/2} w(x) \le C\lambda^{-p}\|f\|^p_{L^p(w)} = C\lambda^{-p}\ \ .
\end{equation}
On $E^c$, for any $\epsilon >0$ one has
$$[V_r f(x)]^{r/2} \le C\Big(\sum_{n < 0} (2^n \lambda)^r M_{2^n\lambda}(x)\Big)^{1/2}$$
$$\le C \sum_{n<0} 2^{n(1-\epsilon)r/2} \lambda^{r/2} M_{2^n\lambda}(x)^{1/2} \ \ \text{any $\epsilon>0$} \ \ .$$
By triangle inequality, it follows that
$$w(\{V_r f> \lambda\} \cap E^c) \le \lambda^{-pr/2}\|1_{E^c} (V_r f)^{r/2} \|_{L^p(w)}^p$$
$$ \le C \lambda^{-pr/2}\Big(\sum_{n<0} 2^{n(1-\epsilon)r/2} \lambda^{r/2} \|M_{2^n\lambda}^{1/2}\|_{L^p(w)}\Big)^p$$
$$ \le C \lambda^{-pr/2}\Big(\sum_{n<0} 2^{n(1-\epsilon)r/2} \lambda^{r/2} (2^n\lambda)^{-1}\|f\|_{L^p(w)}\Big)^p \ \ \text{(using \eqref{e.r2})}$$
Choosing $\epsilon>0$ small one can ensure that $(1-\epsilon)r/2 > 1$. It follows that
\begin{equation}\label{e.small}
w(\{V_r f> \lambda\} \cap E^c)  \le C \lambda^{-pr/2} \Big(\lambda^{r/2}\lambda^{-1}\|f\|_{L^p(w)}\Big)^p  =  C \lambda^{-p} \ \ .
\end{equation}
The desired estimate now follows from \eqref{e.exception} and \eqref{e.small}.

We now show \eqref{e.r2}. Fix $\lambda >0$. It suffices to show that for $N_0(x) < N_1(x) < \dots $ we have
$$\Big\|\big(\sharp\{k: |\sum_{N_{k-1}(x) < j \le N_k(x)} (\Delta_j f)(x)|>\lambda \}\big)^{1/2}\Big\|_{L^p(w)} \le C\lambda^{-1} \|f\|_{L^p(w)} \ \ ,$$
furthermore by a standard argument (see for instance \cite{MR1019960} or \cite{MR2434308}) one can assume that $N_k(x)$ are stopping times with respect to the dyadic martingale in $\mathbb R$. Here, a function $N(x)$ is a stopping time if the level set $\{x: N(x) = k\}$ is an union of standard dyadic intervals of length $2^{-k}$. With this assumption, we'll show the following stronger estimate
$$\Big\|\big(\sum_{k} |\sum_{N_{k-1}(x) < j \le N_k(x)} (\Delta_j f)(x)|^2\big)^{1/2}\Big\|_{L^p(w)} \le C\|f\|_{L^p(w)} \ \ ,$$
and by randomization it suffices to prove for any random sequence $\epsilon_k = \pm 1$:
\begin{equation}\label{e.haarmuliplierweight}
\Big\|\sum_{k\ge 1} \epsilon_k \sum_{N_{k-1}(x) < j \le N_k(x)} (\Delta_j f)(x)\Big\|_{L^p(w)} \le C\|f\|_{L^p(w)} \ \ .
\end{equation}
Take any $k\ge 1$. Let $\T_{\le k}$ be the set of dyadic intervals $I$ such that \\
(i) $N_k$ is constant on $I$, and\\
(ii) for any $x\in I$ the interval $I$ has length at most $2^{-N_k(x)}$.\\
 By the stopping time property of $N_k$ and by the increasing property of $N_k$'s, it is clear that $\T_{\le k}\subset \T_{\le k-1}$, and define
$$\T_k = \T_{\le k-1} \setminus \T_{\le k} \ \ .$$ 
One now writes
$$\sum_{N_{k-1}(x) < j \le N_k(x)} (\Delta_j f)(x) = \sum_{I \in \T_k} \<f, h_I\> h_I(x)$$
and \eqref{e.haarmuliplierweight} follows from boundedness of the martingale transform in the $A_p$ setting (cf. \cite{MR1748283}).
\endproof

\section{Proof of Proposition~\ref{p.restrictedweaktype}}
Without loss of generality assume $w(F)>0$ and $w(G)>0$ and furthermore $\max(w(F),w(G))=1$. The major subsets will be defined using the weighted dyadic  maximal function
\begin{equation*}
M _{w} f (x) := \sup _{I \;:\; x\in I} w (I) ^{-1} \int _{I} \lvert  f\rvert \; w (dx) \,.  
 \end{equation*}
$M _{w}$ bounded from $L^{1,\infty}(w)$ to $L^1(w)$, for any weight, with norm $1$.   \\

\noindent \underline {Case 1: $w(F) \le w(G)$}. It follows that $w(G)=\max(w(F),w(G))=1$. We define $\widetilde F=F$ and 	 
\begin{equation*}
\widetilde G := G \setminus \Bigl\{ M _{w} \mathbf 1_{F} > C w (F) \Bigr\}
\end{equation*}
for some large constant $C$. Assume without loss of generality that $I_P \cap \widetilde F \ne\emptyset$ where  <$ P \in \P$.  Thus, by  Lemma~\ref{l.size-bound} we have
\begin{equation}\label{e.sizebound}
\sigma:=\size(\P) \le C \min\Bigl(1, w (F)^{1/q}\Bigr)\,.
\end{equation}
Let $\tau=w(F)^{1/2q}$. By recursive applications of Lemma~\ref{l.treeselect-bysize} and Lemma~\ref{l.treeselect-bydensity}, we can divide $\P = \bigcup_{n\in \mathbb Z}\P_n$ such that $\P_n = \bigcup_{T\in\T_n} T$ is an union of trees satisfying:
\begin{gather*}
\sum_{T\in \T_n} w(I_T) \le 2^n 
\\
\size(\P_n) \le C \min(\sigma, 2^{-n/(2q)} \tau)\, , 
\\
\density(\P_n) \le C \min(1,2^{-n/r'}) \,.
\end{gather*}
Applying  the tree estimate \eqref{e.tree-est} (with $s=1$), we have
\begin{align*}
B_{\P}(f, g) &\le C\sum_{ n\in \mathbb Z } \sum_{T \in \T_n} w(I_T)\size(T)\density(T)
\\
&\le C { \sum_ {n \in \mathbb Z } 2^n  \min (\sigma, 2^{-\frac n{2q}} \tau) \min(1, 2^{-n/r'})}
\end{align*}
We show that for any $2q/r<\eta<1$ we have
\begin{equation}\label{e.twosidedsum}
\sum_n 2^n \min \Big(\sigma , 2^{-n/(2q)}\tau\Big) \min\Big(1, 2^{-n/r'} \Big) \le C \sigma ^{1-\eta} \tau^{\eta} \,.
\end{equation}
This will imply the desired bound \eqref{e.restrictedweaktype} for $B_{\P}(f, g)$, as one can select $\eta$ very close to $2q/r$ and use \eqref{e.sizebound} to obtain 
$$B_{\P}(f, g1_{\widetilde G})\le C \sigma ^{1- \eta } \tau ^{\eta } 
\le C w (F) ^{(1/q) (1- \eta /2)} \le C w (F) ^{1/p} = C w(F)^{1/p}w(G)^{1/p'}$$
for any $p$ such that {$\frac 1 p < \frac 1 q - \frac 1 r$}. Here, we used the assumption that $w (F) \le 1$. 

It remains to show \eqref{e.twosidedsum}. Take any $\alpha,\beta\in [0,1]$, we estimate the left hand side of \eqref{e.twosidedsum} by
\begin{align*}
&\le C \sigma  \sum_n 2^n \min\Big(1, 2^{- n/(2q)}\frac{\tau}{\sigma}\Big)^{\alpha} \min\Big(1, 2^{-n/r'}\Big)^{\beta}
\\
&\le C \sigma \sum_n 2^n \min\Big(1, 2^{-\alpha n/(2q)}2^{- \beta n/r'} \Bigl(\frac \tau \sigma \Bigr)^{\alpha}  \Big)
\end{align*}
The condition {$r>2q$} ensures that there exists $\alpha,\beta \in [0,1]$ satisfying
\begin{equation}\label{e.lowerconstraint}
\frac \alpha{2q} + \frac{\beta}{r'}>1 \,.
\end{equation}
If $\alpha,\beta$ are such, the last estimate is a two sided geometric series, so is controlled by the largest term, which is about the size of
$$C \sigma  \Bigl(\frac \tau \sigma \Bigr)^{\frac \alpha{\beta/(r')+\alpha/(2q)}} = C \sigma ^{1-\eta}\tau^{\eta} 
\qquad \eta := \frac \alpha{\beta/(r')+\alpha/(2q)} \,. $$
Varying $\alpha,\beta$ in $[0,1]$ respecting \eqref{e.lowerconstraint}, one can get any $\eta \in (\frac{2q}{r},1)$. \\

\noindent \underline{Case 2: $w (F) > w (G)$}. It follows that $w(F)=\max(w(F),w(G))=1$. We choose $\widetilde {G} = G$ and 
$$\widetilde F = F \setminus \Big\{M_w(1_G)> C w(G) \Big\}$$
for some large constant $C$. It follows that
$$\density(\P) \le C w(G)^{1/r'} \ \ .$$
while clearly $\size(\P) \le C$. By recursive applications of  Lemma~\ref{l.treeselect-bysize} and Lemma~\ref{l.treeselect-bydensity} we decompose
$\P = \bigcup_{n \in \mathbb Z}\P_n$ such that $\P_n = \bigcup_{T\in \T_n}T$ is a union of trees satisfying
\begin{gather*}
\sum_{T\in \T_n} w(I_T) \le 2^n 
\\
\size(\P_n) \le C 2^{-\frac n{2q}} 
\\
\density(\P_n) \le C 2 ^{-n/r'} w(G)^{1/r'} \ \ .
\end{gather*}

We now use Lemma~\ref{l.tree-redecomp} and decompose $\P_n$ into $\bigcup_{k\ge 0}\P_{n,k}$ such that each $\P_{n,k} = \bigcup_{T\in \T_{n,k}} T$ is a union of trees, with
$$\size(\P_{n,k}) \le C 2^{-(n+k)/(2q)} \ \ ,$$
$$\|\sum_{T\in \T_{n,k}} 1_{I_T} \|_{L^p(w)} \le C 2^{n+k}  w(F)^{1/p} = C 2^{n+k}\ \ ,$$
$$\|\sum_{T\in \T_{n,k}} 1_{I_T} \|_{L^1(w)} \le C \sum_{T\in \T_n} w(I_T) \le 2^n\ \ .$$
By interpolation of the last two estimates (use $p$ large in the first), we obtain
\begin{equation}\label{e.largep}
\|\sum_{T\in \T_{n,k}} 1_{I_T} \|_{L^{p-\epsilon}(w)} \le C 2^{k/p'} 2^{n} \ \ .
\end{equation}
It follows that 
\begin{align*}
B_{\P_{n,k}}(f, g) &= \int \sum_{T\in \T_{n,k}} 1_{I_T} \sum_{P \in T} \<f,\phi_{P_1}\>\phi_{P_1}a_P(x) g(x) w(x) dx
\\
&\le C \int \Big(\sum_{T \in \T_{n,k}} 1_{I_T}\Big)^{1/r}  \Big(\sum_{T\in \T_{n,k}}|\sum_{P \in T} \<f,\phi_{P_1}\>\phi_{P_1}a_P(x) g(x)|^{r'}\Big)^{1/r'} w(x)dx 
\end{align*}
For $p$ very large we estimate this by
\begin{align}\label{e.BPnk}
\le C\Big\|\big(\sum_{T \in \T_{n,k}} 1_{I_T}\big)^{1/r} \Big\|_{L^{p-\epsilon}(w)} \Big\|\big(\sum_{T\in \T_{n,k}}|\sum_{P \in T} \<f,\phi_{P_1}\>\phi_{P_1}a_P(x) g(x)|^{r'}\big)^{1/r'} \Big\|_{L^{(p-\epsilon)'}(w)}
\end{align}
We'll choose $p$ very large such that $p-\epsilon > r$. Since the function inside the $L^{(p-\epsilon)'}(w)$ norm is supported in $G$, by H\"older's inequality we can estimate the second factor by
$$\le w(G)^{1/(p-\epsilon)' - 1/r'}\Big\|\big(\sum_{T\in \T_{n,k}}|\sum_{P \in T} \<f,\phi_{P_1}\>\phi_{P_1}a_P(x) g(x)|^{r'}\big)^{1/r'} \Big\|_{L^{r'}(w)}$$
$$= w(G)^{1/(p-\epsilon)' - 1/r'}\Big(\sum_{T\in \T_{n,k}}\big\|\sum_{P \in T} \<f,\phi_{P_1}\>\phi_{P_1}a_P(x) g(x)\big\|_{L^{r'}(w)}^{r'} \Big)^{1/r'}$$
and using the tree estimate \eqref{e.tree-est} we can estimate the above expression by
$$\le w(G)^{1/(p-\epsilon)' - 1/r'} \Big(\sum_{T\in \T_{n,k}}w(I_T) \Big)^{1/r'} \size(\P_{n,k}) \density(\P_{n,k})$$
$$\le w(G)^{1/(p-\epsilon)' - 1/r'} 2^{n/r'} 2^{-(n+k)/(2q)} \min\big(2 ^{-n/r'} w(G)^{1/r'}, \density(\P)\big)  \ \ .$$
Since $\density(\P) \le C w(G)^{1/r'}$,  the above expression is controlled by
$$\le C w(G)^{1/(p-\epsilon)'} 2^{-(n+k)/(2q)} \min(1,2^{n/r'})  \ \ .$$
Using \eqref{e.largep}, we obtain an estimate for the first factor in \eqref{e.BPnk}:
$$\Big\|\big(\sum_{T \in \T_{n,k}} 1_{I_T}\big)^{1/r} \Big\|_{L^{p-\epsilon}(w)} = \Big\|\sum_{T \in \T_{n,k}} 1_{I_T} \Big\|_{L^{(p-\epsilon)/r}(w)}^{1/r}$$
$$\le C 2^{n/r} 2^{(1/r-1/p)k} \ \ .$$
Therefore
$$B_{\P_{n,k}}(f, g) \le C w(G)^{1/(p-\epsilon)'} 2^{n/r} 2^{(1/r-1/p)k}  2^{-(n+k)/(2q)} \min(1,2^{n/r'})  \ \ .$$
Note that {$r>2q$} by given assumption, so we always have
$$\frac 1 r < \frac 1 p + \frac 1 {2q} \ \ .$$
Then summing over $k\ge 0$, we obtain
$$B_{\P_n}(f,g) \le C w(G)^{1/(p-\epsilon)'} 2^{n(1/r-1/(2q))} \min(1,2^{n/r'})  \ \ .$$
Finally, summing over $n\in \mathbb Z$ we obtain
$$B_{\P}(f,g) \le C w(G)^{1/(p-\epsilon)'} \sum_{n\in\mathbb Z} 2^{n(1/r-1/(2q))} \min(1,2^{n/r'})$$
and this is a two-sided geometric series and it converges since {$1/r - 1/(2q) <0$} and $1/r + 1/r' - 1/(2q) > 0$. Thus, the series is dominated by its largest term, which is about the size of
$$Cw(G)^{1/(p-\epsilon)'}$$
Since $0< w(G)< 1$ and since we can choose $p<\infty$ arbitrarily large, it follows that for any finite $ p$, 
$$B_{\P}(f,g) \le C w(G)^{1/p'} = C w(F)^{1/p} w(G)^{1/p'}\,   \ \ ,$$
and this completes the proof of Proposition~\ref{p.restrictedweaktype}.

\end{document}